\documentclass[11pt,leqno]{amsart}
\usepackage{amscd,amsfonts}
\usepackage{amssymb}
\usepackage{amsmath}
\usepackage{amsthm}
\usepackage{dcpic,pictex}
\theoremstyle{plain}
\newtheorem{thm}{Theorem}[subsection]
\newtheorem{cor}[thm]{Corollary}
\newtheorem{lem}[thm]{Lemma}
\newtheorem{prop}[thm]{Proposition}

\theoremstyle{definition}
\newtheorem{defn}[thm]{Definition}
\newtheorem{asum}[thm]{Hypothesis}
\newtheorem{ex}[thm]{Example}
\newtheorem{rem}[thm]{Remark}

\newcommand{\V}{\mathcal{V}}

\newcommand{\rank}{\textrm{rank}}

\newcommand{\Ext}{\operatorname{Ext}}
\newcommand{\Hom}{\operatorname{Hom}}
\newcommand{\Ann}{\operatorname{Ann}}
\newcommand{\MaxSpec}{\operatorname{MaxSpec}}
\newcommand{\Ker}{\textrm{Ker}}

\newcommand{\dive}{\operatorname{div}}

\newcommand{\resstar}{\ensuremath{\operatorname{res}^{*}}}
\newcommand{\res}{\ensuremath{\operatorname{res}}}
\newcommand{\0}{\bar 0}
\newcommand{\1}{\bar 1}

\newcommand{\HH}{\operatorname{H}}
\newcommand{\gl}{\mathfrak{gl}}

\newcommand{\Lie}{\ensuremath{\operatorname{Lie}}}
\numberwithin{equation}{subsection}

\def\Z{{\mathbb Z}}

\def\:{\colon}

\def\b{\mathfrak{b}}
\def\t{\mathfrak{t}}

\def\la{\lambda}

\newcommand{\fg}{\mathfrak{g}}
\def \ff{\mathfrak{f}}
\def \ft{\mathfrak{t}}

\def \fh{\mathfrak{h}}
\def \fe{\mathfrak{e}}
\def \fa{\mathfrak{a}}
\def \fh{\mathfrak{h}}

\def \fa{\mathfrak{a}}
\def \sl{\mathfrak{sl}}

\def\La{\mathfrak{g}}

\def\C{{\mathbb C}}

\begin{document}
\title[On cohomology and support varieties for  Lie superalgebras]
{On cohomology and support varieties for  Lie superalgebras}\
\author{Irfan Bagci}
  \date{\today}

\address{Department of Mathematics \\
          University of California at  Riverside\\ Riverside, California 92521  }
\email{irfan@math.ucr.edu}


\begin{abstract}Support varieties for Lie superalgebras over the complex numbers were introduced in \cite{BKN1} using the relative cohomology. In this paper we discuss finite generation of the relative cohomology rings for Lie superalgebras, we formulate a definition for  subalgebras which detect the cohomology,
 also  discuss realizability of support varieties. In the last section as an application we compute the relative cohomology ring of the Lie superalgebra
$\overline{S}(n)$ relative to the graded zero component $\overline{S}(n)_0$ and show that this ring is finitely generated.
We also compute support varieties of all simple modules in the category of finite dimensional $\overline{S}(n)$-modules
which are completely reducible over $\overline{S}(n)_0$.
\end{abstract}

\maketitle

\parskip=2pt

\section{Introduction}

\subsection{}Throughout the present article we work with the complex numbers $\C$ as the ground field. All vector spaces are assumed to be finite dimensional unless otherwise is noted. For a $\Z_2$-graded vector space $V= V_{\0} \oplus V_{\1}$, $\overline{v}$ will denote the $\Z_2$ degree of homogeneous element $v \in V$. A Lie superalgebra is a finite dimensional
$\Z_2$-graded vector spaces $\La=\La_{\0}\oplus \La_{\bar{1}}$ with a bracket
 $[ , ] : \La\otimes \La \rightarrow  \La$ which preserves the $\Z_2$-grading
and satisfies graded versions of the operations used to define Lie algebras.
 The even part  $\La_{\0}$ is a Lie algebra under the bracket. We view a Lie algebras as a Lie superalgebra concentrated in degree $\0$.  A Lie superalgebra $\fg$ will be called \emph{classical} if there is a connected reductive algebraic group $G_{\0}$ such that $\Lie(G_{\0}) = \fg_{\0}$ and an action of $G_{\0}$ on $\fg_{\1}$ which differentiates to the adjoint action of $\fg_{\0}$  on $\fg_{\1}$.

Given a Lie superalgebra ${\mathfrak g}$, $U(\La)$ will denote the universal enveloping algebra of $\fg$. $U(\fg)$ is an associative $\Z_2$-graded algebra. The category of $\La$-modules has objects that are left $U(\La)$-modules
which are $\Z_2$-graded and the action of $U(\La)$ preserves the $\Z_2$-grading.   Morphisms are  described in \cite[\S 2.1]{BKN1}. The category of $\La$-modules is not an abelian category. However
the category of  graded modules consisting of the same objects but with ${\mathbb Z}_2$-graded morphisms is
an abelian category. The parity change functor, $\Pi $, which interchanges the $\Z_2$-grading of a module, allows one to
make use of the standard tools of homological algebra. Since $U({\mathfrak g})$ is a Hopf algebra, one can use the antipode
and coproduct of $U(\La)$ to define a $\La$-module structure on the dual of a module  and the
tensor product of two modules.
All submodules are assumed to be graded.  \emph{Superdimension} of a $\fg$-module $M =M_{\0} \oplus M_{\1}$ is defined to be  $\dim M_{\0} - \dim M_{\1}$.

 A ${\mathfrak g}$-module $M$ is
\emph{finitely semisimple} if it is isomorphic to a direct sum of finite dimensional simple ${\mathfrak g}$ submodules.
Let ${\mathfrak t}$ be a Lie subsuperalgebra of ${\mathfrak g}$. Let $\mathcal{C}_{(\La, \t)}$ be the full subcategory of the category of all $\La$-modules which are finitely
semisimple as $\t$-modules. The category $\mathcal{C}_{(\La, \t)}$ is closed under arbitrary direct sums,
quotients, and finite tensor products (cf. \cite[3.1.6]{Kum}). For $M,N$ in $\mathcal{C}_{(\La, \t)}$,
$\Ext_{\mathcal{C}_{(\fg, \t)}}^d(M,N)$ will denote  the degree $d$ extensions between $N$ and $M$ in $\mathcal{C}_{(\fg, \t)}$.

 For background  on Lie superalgebras and relative cohomology see \cite{Kac} and
 \cite[\S2]{BKN1}, respectively.

 \subsection{} Let $\fg$ be a Lie superalgebra and $\ft$ be a Lie subsuperalgebra of $\fg$.  
 In Section 2   we review the relative cohomology for Lie superalgebras and record
the properties we are going to need in the rest. After that we give necessary and sufficient  conditions for  identify  the relative cohomology  $\HH^\bullet(\La,\t;\C)$ with the cochains and  discuss when the ring of invariants will be finitely generated. We also introduce the notion of a detecting subalgebra  for a Lie superalgebra.

In section $3$ under the assumption of the finite generation of $\HH^\bullet(\La,\t;\C)$ we define support varieties.  the rned with  varieties for Lie superalgebras. Here we associate varieties which are called support varieties to modules for Lie superalgebras. These varieties are affine conical varieties. We prove a theorem which is called realization theorem in the theory of support varieties.
This theorem basically says that we can realize any conical subvariety as variety of some module.

In Section $3$ we give an application. Here using   the work  done for the Cartan type Lie superalgebra $W(n)$ in \cite{BAKN}, we compute the relative cohomology ring for the Lie superalgebra $ \bar{S}(n)$ relative to the degree zero component $ \bar{S}(n)_0$ which is isomorphic to $\gl(n)$ as a Lie algebra. In particular we show that the cohomology ring  is a polynomial ring and $\bar{S}(n)$ admits a detecting subalgebra. By using finite generation theorem we defined support varieties for  $\fg$-modules. By using the fact that a simple module of $W(n)$ is typical if and only if its superdmension is zero and results of Serganova on representations of $\fg$,
we were able to compute support varieties of all finite dimensional simple  modules  which are completely reducible over  $\fg_0$. We also show that the realization theorem proven in Section $1$ holds for $\fg$.

\section{Cohomology}

\subsection{\bf Relative cohomology for Lie superalgebras.} Let us recall the definition of relative cohomoloy
for Lie superalgebras.
Let $\La$ be a Lie superalgebra and let $\t\subseteq \La$ be a Lie subsuperalgebra.
 Let $M$ be a $\La$-module. The cochain complex $C^{\bullet}(\fg, M)$ is defined by
 $$C^{\bullet}(\fg, M) = \bigoplus_{p \   \in \ \Z_{\geq 0}}C^p(\fg, M) \ \ \text{with} \ \  C^p(\La, M)=\Hom(\Lambda _s^p(\La), M). $$
 Here $\Lambda _s^p(\La)$ denotes the super wedge product, i.e,  the space of $\Z_2$-graded $p-$ alternating tensors on $\fg$.   $(\Lambda _s^p(\La))_{r}$, $r \in \Z_2$, is spanned by
 $$x_1 \wedge x_2\wedge \cdots \wedge x_p  \ \ (x_j \ \in  \fg)$$
 satisfying $\Sigma_{j=1}^p\overline{x_j} = r$ and

 $$x_1 \wedge x_2\wedge \cdots \wedge x_j \wedge x_{j+1} \wedge \cdots \wedge x_p= -(-1)^{\overline{x_j}  \ \overline {x_{j+1}}}x_1 \wedge x_2\wedge \cdots \wedge x_{j+1} \wedge x_{j} \wedge \cdots \wedge x_p.$$
 Thus $x_j, x_{j+1}$ skew commute unless both are odd in which case they commute.

The differential (or coboundray operator) $d$ is the even linear map of $\Z$ degree one  satisfying
{\setlength\arraycolsep{2pt}\begin{eqnarray}\label{eq:cohom}d(\phi )(x_1\wedge  ... \wedge x_{p+1})&=&\sum_{1\leq i<j \leq p+1}(-1)^
{i+j+(\overline{x_i}+\overline{x_j})k_i+\overline{x_j}k_{i,j}}\phi ([x_i, x_j] \ \wedge x_1\wedge \  ...\  \wedge \hat {x_i}\wedge  ... \wedge
\hat {x_j}\wedge ...\wedge x_{p+1}) \nonumber \\
& & \quad +\sum _i(-1)^{j+1+\overline{x_j}(k_j+\overline{\phi})}x_i\phi (x_1\wedge  ... \wedge \hat {x_i}\wedge  ... \wedge x_{p+1}),
\end{eqnarray}

where $x_1,...,x_{p+1}$ and $\phi$ are assumed to be homogeneous, and
$k_i:=\Sigma_{j=1}^{i-1}\overline{x_j}$,  $k_{i,j}:= \Sigma_{s=i+1}^{j-1}\overline{x_s}.$

We denote the restriction of $d$ to $C^p(\fg, M)$ by $d^p$. Since $d^2 =0$, i.e., $d^p\circ d^{p-1}=0$ for all $p\in \Z_{\geq 0}$,   \ ($C^\bullet(\fg, M), d$) is a chain complex.

Then we define

$$\HH^p(\La; M)=\Ker \  d^p/Im \  d^{p-1}.$$

The relative version of the above construction is defined as follows:  Let $\La, \ \t$, and $M$ be as above.
Define

$$C^p(\La, \t;M)=\Hom_\t(\Lambda _s^p(\La/\t),M).$$

Then the map $d^p$ gives a map $d^p:C^p(\La,\t;M) \rightarrow C^{p+1}(\La,\t;M)$ and we define
$$\HH^p(\La,\t;M)=\Ker \ d^p/Im \ d^{p-1}.$$

 For $U(\fg)$ modules $M, N$ we can  define cohomology for  the pair $(U(\fg),U({\mathfrak t}))$  which we denote by $\Ext_{(U(\La),U(\t))}^{\bullet}(M,N)$ (cf.\ \cite[Section 2.2]{BKN1}). If $\fg$ is finitely semisimple as a $\ft$
module under the adjoint action then for $M, N \in \mathcal{C}_{(\fg, \ft)}$ we have the following important isomorphisms
$$\Ext_{\mathcal {C}_{(\fg, \ft)}}^{\bullet}(M,N)\cong \Ext_{(U(\La),U(\t))}^{\bullet}(M,N) \cong  \Ext_{(U(\La),U(\t))}^\bullet (\C,M^\ast \otimes N)\cong \HH^\bullet (\La,\t;M^\ast \otimes N).$$

\subsection{Relating cohomology rings to invariants}

In this subsection we give necessary and sufficeint  conditions to identify the relative cohomology ring $\HH^{\bullet}(\fg, \t; \C)$ with the cochains and discuss finite generation of the relative cohomology rings.

\begin{prop}\label{P:cohomologyring} Let $\t$ be a Lie subsuperalgebra of $\fg$. Then $\HH^{\bullet}(\fg, \t; \C)$ is isomorphic to the ring of invariants $ \Lambda_s^{\bullet}(\ft^\ast)^{\t}$ if
and only if $[\fg, \fg]\subseteq \ft$.
\end{prop}
\begin{proof}

Recall that the cochains are defined by
\begin{equation*} \label{E:1}
C^p(\fg, \t; \C) =  \Hom_{\t}(\Lambda_s^p(\fg /\t), \C)\cong \Lambda_{s}^{p} \left(\ft^{*}  \right)^{\t}.
\end{equation*}

If $\HH^{\bullet}(\fg, \t; \C)\cong \Lambda_s^{\bullet}(\ft^\ast)^{\t}$, then since cohomology  is a subquotient of cochains all differentaials has to be zero. If $[\fg, \fg]$ is not contained in $\ft$ one can always construct a linear functional $\phi : \fg / \ft \rightarrow \C$ which does not go to zero under the differential $d$.

Suppose conversely that $[\fg, \fg]\subseteq \ft$. Now note that in this case the differential $d^p$ in \eqref{eq:cohom} is identically zero. The  second sum of \eqref{eq:cohom} is zero since here $M = \C$ and since
each $[x_i, x_j]$ is zero in the quotient $\fg/ \t$ the first sum of \eqref{eq:cohom}is zero as well.
Therefore
$$\HH^{\bullet}(\fg, \t; \C)\cong \Lambda_s^{\bullet}((\fg/ \ft)^\ast)^{\t}.$$

\end{proof}

\begin{cor}\label{T:Gfinitegeneration} 
Assume that $\HH^{\bullet}(\fg, \t; \C)\cong \Lambda_s^{\bullet}((\fg/ \ft)^\ast)^{\t}$ and  assume  in addition that $\t$ is a reductive Lie algebra. Let $G$ be  the connected reductive algebraic group such that $\Lie(G)=\t$. Let $M$ be a finite dimensional $\fg$-module. Then,
\begin{itemize}
 \item[(a)] The superalgebra $\HH^{\bullet}(\fg, \t; \C)$ is finitely generated as a ring.
\item[(b)] $ \HH^{\bullet}(\fg, \t; M)$ is finitely generated as an $\HH^{\bullet}(\fg, \t; \C)$-module.
\end{itemize}
\end{cor}
\begin{proof}\begin{itemize}

\item [(a)] Since $\HH^{\bullet}(\fg, \t; \C) \cong \Lambda_s^{\bullet}((\fg/ \ft)^\ast)^{\t} =\Lambda_s^{\bullet}((\fg/ \ft)^\ast)^G$ and $G$ is reductive this statement follows from  the classical invariant theory
result of Hilbert \cite[Theorem 3.6]{PV}.

\item[(b)] Since $M$ is finite dimensional
\begin{align*}
\Hom_{\C}(\Lambda_s^{\bullet}(\fg/\t),M)&\cong \Lambda_s^{\bullet} ((\fg/\t)^{*})\otimes M \\
\end{align*}
is finitely generated as a $\Lambda_s^
{\bullet}((\fg/ \t)^\ast)$-module and by
  \cite[Theorem 3.25]{PV}
\[
\Hom_\C(\Lambda_s^{\bullet}(\fg/\t),M)^{G}=\Hom_{G}(\Lambda_s^{\bullet}(\fg/\t),M) = C^{\bullet}(\fg ,\t;M)
\] is finitely generated as a $\Lambda_{s}^{\bullet}((\fg / \t)^{*})^{G}\cong \HH^\bullet(\fg,\t;\C)$-module. Now we can argue as in \cite[Theorem 2.5.3]{BKN1} to deduce that $\HH^{\bullet}(\fg, \t; M)$ is finitely generated as a $\HH^{\bullet}(\fg, \t; \C)$-module.

\end{itemize}
\end{proof}
\begin{ex}

Suppose that $\t=\fg_{\0}$. Since  super wedge product is symmetric product on odd spaces and $[\fg_{\1}, \fg_{\1}]\subseteq \fg_{\0}$ we have the following important isomorphism
$$\HH^{\bullet}(\fg, \fg_{\0}; \C)\cong \Lambda_s^{\bullet}((\fg/ \fg_{\0})^\ast)^{\fg_{\0}} \cong  S(\fg_{\1}^\ast)^{\fg_{\0}}.$$

Recall that a Lie superalgebra $\fg = \fg_{\0} \oplus \fg_{\1}$ is called classical if there is a connected reductive algebraic group $G_{\0}$ such that $\Lie(G_{\0}) = \fg_{\0}$ and an action of $G_{\0}$ on $\fg_{\1}$ which differentiates to the adjoint action of $\fg_{\0}$  on $\fg_{\1}$.
Note that if $\fg$ is a classical Lie superalgebra then $\HH^{\bullet}(\fg, \fg_{\0}; \C)$ is always finitely generated.

\end{ex}

\begin{asum} \label{A1} Throughout the rest of this section and next section we fix a pair $(\fg, \ft)$ and assume that  $\HH^{\bullet}(\fg, \t; \C)$ is  finitely generated.
\end{asum}

\subsection{Detecting subalgebras}

In \cite{BKN1, BKN2} by applying invariant theory
results in \cite{LR} and \cite{dadokkac} the authors showed that under suitable conditions  a classical Lie superalgebra $\fg = \fg_{\0} \oplus \fg_{\1}$ admits a subalgebra
$\fe=\fe_{\0}\oplus \fe_{\bar{1}}$
 such that the restriction map in cohomology induces an isomorphism
\[
\text{H}^{\bullet}(\fg,\fg_{\0};{\mathbb C})
\cong \text{H}^{\bullet}(\fe,\fe_{\0};{\mathbb C})^{W},
\]
where $W$ is a finite pseudoreflection group.

Similar subalgebras were constructed for the Cartan type Lie superalgebra $W(n)$ in \cite{BAKN}. The goal of all this work is to construct subalgebras
that can play the role  of elementary abelian groups in the theory of support varieties for finite groups.

\begin{defn}
Let $\fe=\fe_{\0}\oplus \fe_{\1}$ be a  subalgebra of $\fg$ such that
\begin{itemize}
\item [(a)] $\fe$ is a classical Lie superalgebra, and
 \item [(b)] the inclusion map $\fe \hookrightarrow \fg$ induces an isomorphism
$$\HH^{\bullet}(\fg, \t; \C) \cong \HH^{\bullet}(\fe, \fe_{\0}; \C)^W$$
for some group $W$ such that $(-)^W$ is exact.

\end{itemize}
A subalgebra with these properties will be called a  \emph{detecting subalgebra} for the pair $(\fg, \t)$.
\end{defn}

\begin{rem} Observe that in the definition above we are implicitly assuming that $\fe_{\0} \subseteq \ft_{\0}$.
\end{rem}

\begin{ex} Let $\fg=\gl(m|n)  $ be
the Lie superalgebra of $(m+n) \times (m+n)$ complex matrices with $\Z_2$-grading given by
$$\overline{E_{i,j}} = \0  \ \ \text{if} \  1\leq i, j\leq m \  \text{or} \   m+1\leq i, j\leq m+n \ \ \text{and} \
\overline{E_{i,j}} = \1 \ \ \text{otherwise},$$
where $E_{i,j}$  denotes the $(i,j)$ matrix unit. Let  $\t= \fg_{\0}$. As in \cite[Section 8.10]{BKN1}, one can take
$\fe_{\1}\subseteq \fg_{\1}$ to be the subspace spanned by

$$ E_{m+1-s, m+s}+E_{m+s, m+1-s} \ \ \text{for} \ \ s=1, \dots, r.$$
Let $\fe_{\0}$ be the stabilizer of $\fe_{\1}$ in $\fg_{\0}$. Then $\fe=\fe_{\0}\oplus \fe_{\1}$ is a detecting subalgebra for the pair $(\fg, \ft)$.
\end{ex}


Since a pair $(\fg, \t)$ may not have a detecting subalgebra we make the following technical assumption about our fixed pair $(\fg, \t)$.
\begin{asum}\label{A2}Throughout the next section we assume that our fixed  pair $(\fg, \t)$ has a detecting subalgebra $\fe = \fe_{\0} \oplus \fe_{\1}$.
\end{asum}

\section{Support Varieties}

Recall that throughout this section we assume that $\HH^{\bullet}(\fg, \ft; \C)$ is finitely generated.

 \subsection{} Let $M$ and $N$ be $\fg$-modules such that $ \HH^\bullet (\fg, \ft; \Hom_{\C}(M, N)))$ is
finitely generated as an $\HH^\bullet (\fg,\ft;\C )$-module.

 Let $$I_{(\fg, \ft)}(M, N)= \Ann_{\HH^{\bullet} (\fg,\ft;\C )}(\HH^\bullet (\fg , \ft; \Hom_{\C}(M, N)))$$

be the annihilator ideal of this module. We define the   \emph{relative support variety} of the pair
$(M,N)$ to be
\[
\V_{(\fg , \ft)}(M, N)= \MaxSpec(\HH^{\bullet} (\fg,\ft;\C )/I_{(\fg, \ft)}(M, N)),
\]
the maximal ideal spectrum of the quotient of $\HH^{\bullet} (\fg,\ft;\C )$ by $I_{(\fg, \ft)}(M, N)$.
For short  when $M=N$, write

$$I_{(\fg, \ft)}(M)=I_{(\fg, \ft)}(M, M),$$

$$\V_{(\fg, \ft)}(M)=\V_{(\fg, \ft)}(M,M).$$
We call $\V_{(\fg, \ft)}(M)$ the \emph{support variety} of $M$.


\subsection {}As the detecting subalgebra $\fe = \fe_{\0} \oplus \fe_{\1}$ is classical $\HH^{\bullet}(\fe, \fe_{\0}; \C) \cong S(\fe_{\1}^\ast)^{\fe_{\0}}$ is finitely generated. Therefore one can define support varieties for  $\fe$-modules as in $(3.1)$.

 The canonical restriction map $$\operatorname{res} :\HH^\bullet (\fg,\ft;\C )\rightarrow \HH^{\bullet}(\fe, \fe_{\0}; \C)$$ induces a map of varieties $$\resstar :\V_{(\fe, \fe_{\0})}(\C) \rightarrow \V_{(\fg, \ft)}(\C).$$
By the isomorphism $\HH^{\bullet}(\fg, \ft; \C) \cong \HH^{\bullet}(\fe, \fe_{\0}; \C)^W$ one then has
$$  \V_{(\fe, \fe_{\0})}(\C)/W \cong \V_{(\fg, \ft)}(\C).$$
In particular for any finite dimensional $\fg$-module $M$ , where $ \HH^\bullet (\fg, \ft; \Hom_{\C}(M, M)))$ is finitely generated, $\V_{(\fe, \fe_{\0})}(M)/W$ and $\V_{(\fg, \t)}(M)$ can naturally be viewed as affine subvarieties of $\V_{(\fg, \ft)}(\C)$. 

Furthermore, $\resstar$ restricts to give a map,
$$ \V_{(\fe, \fe_{\0})}(M) \rightarrow \V_{(\fg, \ft)}(M).$$
Since $\V_{(\fe, \fe_{\0})}(M)$ is stable under the action of $W$ we have the following embedding induced by $\resstar$,
\begin{equation}\label{E:Nresstarimage}\V_{(\fe, \fe_{\0})}(M)/W \cong \resstar(\V_{(\fe, \fe_{\0})}(M)) \hookrightarrow \V_{(\fg, \t)}(M).
\end{equation}

\subsection{}
 For a homogeneous $x \in \fe$, let $<x>$ denote the Lie subsuperalgebra generated by $x$. If $M$ is a $\fe$-module then define the \emph{rank variety } of $M$ to be
$$\V_{\fe}^{\rank} = \{ x \in \fg_{\1} \mid M  \ \text{is not projective as a } U(<x>)-\text{module}\} \cup \{0\}$$

\begin{asum}\label{A3} We assume that the detecting subalgebra $\fe$ has a rank variety description; i.e., for any $\fe$-module $M$, $\V_{\fe}^{\rank}(M) \cong \V_{(\fe, \fe_{\0})}(M)$.
\end{asum}

By using the rank variety description   one can prove a number of properties of  $\fe$-support varieties. We record some of these properties and for the proofs and other properties we refer the reader to \cite[Section 6]{BKN1}.
\begin{lem}\label{L:super}Assume that the Hypotheses \ref{A1}, \ref{A2} and \ref{A3} hold for our fixed pair $(\fg, \ft)$. Let $M$ and $N$ be finite dimensional $\fe$-supermodules. Then
\begin{itemize}
\item[(1)]$\V_{(\fe,\fe_{\0})}(M \otimes N) = \V_{(\fe,\fe_{\0})}(M) \cap \V_{(\fe,\fe_{\0})}(N)$.
\item[(2)]$\V_{(\fe,\fe_{\0})}(M \oplus N)= \V_{(\fe,\fe_{\0})}(M) \cup \V_{(\fe,\fe_{\0})}(N)$.
\item[(3)] $M$ is projective if and only if $\V_{(\fe,\fe_{\0})}(M)=\{0\}$.
\item[(4)] If superdimension of $M$ is nonzero, i.e., $\dim M_{\0} \neq \dim M_{\1}$, then
$$\V_{(\fe,\fe_{\0})}(M)=\V_{(\fe,\fe_{\0})}(\C).$$
\end{itemize}

\end{lem}

\subsection{}

Let $0 \neq \zeta \in \HH^{n}(\fg, \ft; \C)$. Since $\HH^{n}(\fg, \ft; \C) \cong \Hom_{\fg}(\Omega^n(\C), \C)$ , where $\Omega^n(\C)$ denotes the $n$th syzygy,
$\zeta$ corresponds to a surjective map $\hat{\zeta}:\Omega^n(\C) \rightarrow \C.$
We  set
$$ L_{\zeta} =  \Ker (\hat{\zeta}:\Omega^n(\C)\rightarrow \C )\subseteq \Omega^n(\C).$$
The modules $ L_{\zeta}$  are often called $\lq \lq$Carlson modules". The importance of the modules $L_{\zeta}$ is that their support variety can be explicitly computed.

\begin{lem}\label{L:esupportzeta} Assume that the Hypotheses \ref{A1}, \ref{A2} and \ref{A3} hold for the pair $(\fg, \ft)$. Let $\zeta \in \HH^{n}(\fg, \ft; \C)$ and $L_{\zeta}$ be as above, then
$$V_{(\fe ,\fe_{\0})}(L_{\zeta}) = \V_{(\fe ,\fe_{\0})}(L_{\res(\zeta)}) = Z(\res(\zeta)).$$
\end{lem}
\begin{proof}This is argued as in \cite[Theorem 6.4.3]{BKN1}.
\end{proof}

\subsection{Realization Theorem}

An important property in the theory of support varieties is the realization of any homogeneous variety as the support variety of a module.
Realizability of support varieties was first proven in \cite[Theorem 6.7]{BKN1} for the detecting subalgebras  of classical Lie superalgebras. This result was later extended to $\fg$ support varieties for classical Lie superalgebras and Cartan type Lie superalgebra $W(n)$ in \cite[Theorem 8.8.1]{BAKN}.
\begin{prop}\label{P:relative} Assume that the Hypothesis \ref{A1} holds for our fixed pair $(\fg, \ft)$.
 Let $\zeta _1,\dots, \zeta _n \in \HH^{\bullet}(\fg, \t; \C)$ be
homogeneous elements with corresponding Carlson modules
$L_{\zeta_1}, \dots, L_{\zeta_n}$. Then
\begin{itemize}
\item[(1)] $\HH^{\bullet}(\fg, \t; L_{\zeta_1}^\ast\otimes
\dotsb  \otimes L_{\zeta_n}^\ast)$ is finitely generated over $\HH^{\bullet}(\fg, \t; \C)$.

\item[(2)] ${\mathcal V}_{(\fg, \t)}(L_{\zeta_1}\otimes
\dotsb \otimes L_{\zeta_n},{\mathbb C})\subseteq  {\mathcal V}_{(\fg, \t)}(L_{\zeta_1},\C) \cap \dots \cap {\mathcal V}_{(\fg, \t)}(L_{\zeta_n}, \C).$
\end{itemize}
\end{prop}

\begin{proof}The proof  is the same  as in \cite[Proposition 8.6.1]{BAKN} and will be skipped.
\end{proof}

\medskip

\begin{lem}\label{L:gsupportzeta} Suppose that the pair $(\fg, \ft)$ satisfies the Hypotheses \ref{A1}, \ref{A2} and \ref{A3}. Let $\zeta \in \HH^{n}(\fg, \ft; \C)$ and let $L_{\zeta}$ be the corresponding  Carlson module.
\begin{itemize}
\item[(1)]If $L_{\zeta}$ is finite dimensional, then $$\V_{(\fg, \ft)}(L_{\zeta})= \resstar(\V_{(\fe, \fe_{\0})}(L_{\zeta}))= Z(\zeta).$$
\item[(2)]If $L_{\zeta}$ is infinite dimensional, then $$\V_{(\fg, \ft)}(L_{\zeta}, \C)= \resstar(\V_{(\fe, \fe_{\0})}(L_{\zeta})) = Z(\zeta).$$
\end{itemize}
\end{lem}
\begin{proof}
\begin{itemize}
\item[(1)]  Since $\resstar(\V_{(\fe, \fe_{\0})}(L_{\zeta}))= Z(\zeta)$ by  Lemma \ref{L:esupportzeta}, we have

$$Z(\zeta) = \resstar(\V_{(\fe, \fe_{\0})}(L_{\zeta})\subseteq \V_{(\fg, \ft)}(L_{\zeta}).$$
For the other containment $\V_{(\fg, \ft)}(L_{\zeta})\subseteq Z(\zeta)$ it is enough to show that some power of $\zeta$ annihilates $\HH^{\bullet}(\fg, \ft; \C)$. One can argue exactly as in \cite[Proposition 6.13]{Ben} to show that $\zeta^2$ annihilates $\HH^{\bullet}(\fg, \ft; \C)$.

\item[(2)] We have
$$\resstar : \V_{(\fe, \fe_{\0})}(L_{\zeta}, \C) \hookrightarrow \V_{(\fg, \ft)}(L_{\zeta}).$$

As in the case of finite groups $L_{\zeta} \cong L_{\res(\zeta)}\oplus P$ as $\fe-$modules where $P$ is some projective $\fe-$module (cf.\cite[Section 5.9]{Ben}).
Observe that by definition $L_{\res(\zeta)}$ can be assumed to lie in the principal block of $\fe$. By \cite[Proposition 5.2.2]{BKN1} we also know that there are no simple modules other than trivial module in the principal block of $\fe$. These observations with \cite[Proposition 5.7.1]{Ben} imply that
$$\V_{(\fe, \fe_{\0})}(L_{\zeta})= \V_{(\fe, \fe_{\0})}(L_{\zeta}, \C).  $$
Since $\resstar(\V_{(\fe, \fe_{\0})}(L_{\zeta}))= Z(\zeta)$ by  Lemma \ref{L:esupportzeta},
$$Z(\zeta) = \resstar(\V_{(\fe, \fe_{\0})}(L_{\zeta}, \C)\subseteq \V_{(\fg, \ft)}(L_{\zeta}, \C).$$
For the reverse containment once again we can use the proof given in \cite[Proposition 6.13]{Ben} to show that $\zeta^2$ annihilates $\HH^{\bullet}(\fg, \ft; L_{\zeta} \otimes L_{\zeta}^{\ast})$.  This also implies that $\zeta^2$ annihilates $\HH^{\bullet}(\fg, \ft; L_{\zeta}).$

\end{itemize}
\end{proof}

 We can now  prove the  realization theorem.
\begin{thm}\label{T:Realization}
\label{T:realization} Suppose that the pair $(\fg, \ft)$ satisfies the Hypotheses \ref{A1}, \ref{A2} and \ref{A3}.

Let  $X$ be a conical subvariety of $\V_{(\fg, \t)}(\C ).$
\begin{itemize}

\item[(a)]If the Carlson modules are finite dimensional for  $\fg $ then there exists a finite dimensional $\fg$-module $M$  such that
$$ \V_{(\fg, \t)}(M)=X.$$
\item[(b)]There exists a $\fg$-module $M$ such that
$$ \V_{(\fg, \t)}(M,\C)=X.$$

\end{itemize}
\end{thm}
\begin{proof}

 Let $J=(\zeta _1, \dotsc , \zeta _n) \subseteq \HH^{\bullet}(\fg, \t; \C)$ be the homogeneous ideal which defines the homogeneous variety $X$. That is,$$X=\mathcal{Z}(\zeta _1)\cap  \dotsb \cap \mathcal{Z}(\zeta _n).$$  Let $M=L_{\zeta _1}\otimes \dotsb \otimes L_{\zeta _n}$.

 Let us also observe that since by the Hypothesis \ref{A2} the group $W$ is exact by arguing as in \cite[Lemma 8.3.1]{BAKN} we can show that
 \begin{equation}\label{E:star}
 \resstar(\V_{(\fe, \fe_{\0})}(L_{\zeta} \otimes L_{\mu})) = \resstar(\V_{(\fe, \fe_{\0})}(L_{\zeta})) \cap \resstar(\V_{(\fe, \fe_{\0})}(L_{\mu})).
 \end{equation}

 for Carlson modules $L_{\zeta}$ and $L_{\mu}$.
\begin{itemize}

\item[(a)]
    Combining (\ref{E:star}),  Lemma  \ref{L:gsupportzeta}  (1) and using the fact that $ \V_{(\fg, \t)}(M_1 \otimes M_2) \subseteq \V_{(\fg, \t)}(M_1) \cap \V_{(\fg, \t)}(M_2)$  we have
\end{itemize}

\begin{align*}
X &=\mathcal{Z}(\zeta _1)\cap  \dotsb \cap \mathcal{Z}(\zeta _n)
  =\V_{(\fg, \t)}(L_{\zeta_{1}}) \cap \dots \cap \V_{(\fg, \t)}(L_{\zeta_{n}} )\\
  &=  \resstar (\V_{(\fe ,\fe_{\0})}(L_{\zeta_{1}} )) \cap \dots \cap \resstar (\V_{(\fe ,\fe_{\0})}(L_{\zeta_{n}} ))
  = \resstar \left(\V_{(\fe,\fe_{\0})}(M) \right)
   \subseteq \V_{(\fg, \t) }\left(M \right) \\
  &  = \V_{(\fg, \t)}(L_{\zeta_{1}}) \cap \dots \cap \V_{(\fg, \t)}(L_{\zeta_{n}} )
   = X.
\end{align*}  It then follows that $\V_{(\fg, \t)}(M)=X.$

\item[(b)] Argued as in (a) by using (\ref{E:star}),  Lemma  \ref{L:gsupportzeta}  (2) and  Proposition ~\ref{P:relative} $(2)$. 

\end{proof}

\section{An Application }

\subsection{The Lie superalgebra $\overline{S}(n)$} We begin by recalling the definition of finite dimensional Lie superalgebras of type $S(n)$. As
a background source we refer the reader to the pioneering paper of Kac \cite{Kac} or the book of M. Scheunert \cite{Sch}.

Let $n$ be a positive integer and assume that $n\geq 2$. Let $V$ be an $n$-dimensional complex vector space and let  $\Lambda(n)$ denote the exterior algebra of $V$.
The exterior algebra $\Lambda(n) = \oplus_{l=0}^n\Lambda(n)_l$ is an associative $\Z$-graded superalgebra. The $\Z_2$-garding is inherited from $\Z$-grading by setting
$\Lambda(n)_{\0} = \oplus \Lambda(n)_{2l}$ and $\Lambda(n)_{\1} = \oplus \Lambda(n)_{2l+1}$. Fix an ordered basis $\xi_1, \dots, \xi_n$ for $V$. For each ordered subset
$I=\{ i_1, i_2,  \dots, i_l\}$ of  $N=\{1, 2, \dots, n\}$ with $i_1< i_2<  \dots< i_l$, let $\xi_I$ denote the product $\xi_1 \xi_2 \dots \xi_l$. The set of all such $\xi_I$ forms a basis of $\Lambda(n)$.

 Then as a super space $W(n)$ is the space  of super derivations of $\Lambda(n)$. $W(n)$ is a Lie superalgebra via supercommutator bracket. The $\Z$-grading on $\Lambda(n)$ induces a $\Z$-grading on $W(n)$
$$W(n)= W(n)_{-1} \oplus W(n)_0 \oplus \dots \oplus W(n)_{n-1},$$
where $W(n)_l$ consists of derivations that increase the degree of a homogeneous element by $l$ and this $\Z-$grading is consistent with the $\Z_2$-grading, i.e.,
$W(n)_{\0} = \oplus W(n)_{2l}$ and $W(n)_{\1} = \oplus W(n)_{2l+1}$. The super commutator bracket preserves the $\Z$-grading on $W(n)$ and thus $W(n)_0$ is a Lie algebra
and each $W(n)_l$ is a $W(n)_0$-module under the bracket action.

Every element of $W(n)$ maps $V$ into $\Lambda(n)$ and since it is a superderivation it is completely determined by its action on $V$. Thus $W(n)$ can be identified with
$\Lambda(n) \otimes V^\ast$ as a vector space.

Denote by  $\partial_i$, $1\leq i \leq n$, the derivation of $\Lambda(n)$ defined by
$$\partial_i(\xi_j) = \delta_{ij}.$$  Then the set of all $\xi_I \otimes \partial_i$ forms a basis of $\Lambda(n) \otimes V^\ast$.
We will write $\xi_I \partial_i$ instead of $\xi_I \otimes \partial_i$. We use the identification above to identify $W(n)$ with $\Lambda(n) \otimes V^\ast$. Under this identification the derivation $\partial_i$ corresponds to the dual of $\xi_i$
and every  element $D$ of $W(n)$ can be uniquely written in the form
$$\sum_{i=1}^nf_i\partial_i,$$
where $f_i \in \Lambda(n)$.

The superalgebra $S(n)$ is the subalgebra of $W(n)$ consisting of all elements $D \in W(n)$ such that
$\dive(D) = 0$, where
$$\dive(\sum_{i=1}^n f_i \partial_i) = \sum_{i=1}^n \partial_i(f_i).$$
The superalgebra $S(n)$ has a $\Z$-grading induced by the grading of $W(n)$
$$S(n)= S(n)_{-1} \oplus S(n)_0 \oplus \dots \oplus S(n)_{n-2}$$
and $S(n)_0$ is isomorphic to $\sl(n)$.

Let $\mathcal{E} =\Sigma _{i=1}^{n}\xi_i\partial _i \in W(n)$.  Note that $\mathcal{E} \notin S(n)$. We shall attach $\mathcal{E}$ to
$S(n)$  and consider the subsuperalgebra $\overline{S}(n)=S(n)\oplus \C \mathcal{E}$ of $W(n)$. The superalgebra $\overline{S}(n)$ admits a $\Z$-grading
$$\overline{S}(n)= \overline{S}(n)_{-1} \oplus \overline{S}(n)_0 \oplus \dots \oplus \overline{S}(n)_{n-2},$$
where $\overline{S}(n)_0 \cong \gl (n)$ as a Lie algebra and $\overline{S}(n)_k  = S(n)_k$ for $k \neq 0$.

\subsection{Notation}We fix the following notations for the rest of the paper. Set $\fg =  \overline{S}(n)$ with $\fg_i = \overline{S}(n)_i$, $i \in \Z$ and $\fg_{\overline{\imath}} =\overline{S}(n)_{\overline{\imath}}$, $\overline{\imath} \   \in \Z_2$. Furthermore, let $\fg^+ = \fg_1\oplus \dots \oplus \fg_{n-2}$. Then
$$\fg = \fg_{-1} \oplus \fg_0 \oplus \fg^+.$$

All $\fg$-modules will be in the category $\mathcal{C}_{(\fg, \fg_0)}$.

\subsection{Basis for $\fg = \overline{S}(n)$} In this subsection we fix a basis for $\fg$ and use this basis in the rest. The $\fg_{-1}$
has basis $\{\partial_i  \mid 1\leq i \leq n\}$ and  $\fg_0$ has basis $\{\xi_i \partial_j \mid 1\leq i, j \leq n\}$.

Let $N=\{1, \dots, n \}$ and let $I$ be an ordered subset of $N$. A spanning set for each $\fg_k$ for $k \neq 0, -1$ can be defined as follows
and contains two distinct types of elements. The elements of type $(I,k)$ are all those of the form $\xi_I\partial _i$
with $i\not\in I$ and $|I| =k+1$. Those of type $(II,k)$ are of the form $\xi _Ah_{ij}$ where $i, j \not\in A$
and $|A| =k$. Here by definition $h_{ij}=\xi _i\partial _i-\xi _j\partial _j$.

The type I elements are all linearly independent, and their span $\fg_k^{I}$ is independent of the span
$\fg_k^{II}$ of the type II elements. The type II elements are not independent however, since $h_{ij}+h_{jk}=h_{ik}$
we reduce the set of type II elements to a basis for $\fg_k^{II}$ as follows. For each $A$ with $|A| =k$, order the complement $B=N - A$
in the natural way as a subset of $N$ and let $i$ be the first element of $B$. Select those element of the form $\xi _Ah_{ij}$ where
$i<j \in B$. These are easily seen to be independent and span $\fg_k^{II}$.

\subsection{Representation theory and atypicality for   $\fg = \overline{S}(n)$}

 Let $\fg= W(n)$ or $\overline{S}(n)$. A Cartan subalgebra $\fh$ of $\fg$ coincides with a Cartan subalgebra of $\fg_0$.  We fix a maximal torus $\fh \subseteq \fg_{0}$ and  a Borel subalgebra $\b_0$ of $\fg_0$ . We will denote by
$X_0^+ $  the parametrizing set of highest weights for the simple finite dimensional $\fg_0$-supermodules with respect to our fixed  pair $(\fh , \b_0)$.  Let $L_0(\la)$ denote the simple finite dimensional $\fg_{0}$-supermodule with
highest weight $\la\in X_{0}^{+}.$ We view $L_{0}(\lambda)$ as a $\fg_{0}$-supermodule concentrated in degree $\0$.

The \emph{Kac supermodule} $K(\la )$ is the induced representation of $\fg$,
$$K(\lambda )=U(\fg)\otimes _{U(\fg_{0}\oplus \fg^{+})}L_0(\la ),$$
where $L_0(\la )$ is viewed as a $\fg_{0}\oplus \fg^{+}$ by letting $\fg^+$ act trivially.
 $K(\la )$ is  finite dimensional and   with respect to the choice of Borel subalgebra $\b_0\oplus \fg^{+} \subseteq \fg$ one has a dominance order on weights.  With respect to this ordering $K(\lambda)$ has highest weight $\la$ and a unique simple quotient which we denote by $L(\la )$.
Conversely, every finite dimensional simple supermodule appears as the head of some Kac supermodule (cf.\ \cite[Theorem 3.1]{Ser}).

From our discussion above we see that  the set
$$\{L(\la)\mid \la \in X_0^+\}$$ is a complete irredundant collection of simple finite dimensional $\fg$-supermodules.

Following Serganova, we call
$\la \in \fh^\ast$ \emph{typical} if $K(\la)$ is a simple module \emph{atypical} otherwise.

   Choose the standard basis $\varepsilon _1,\dotsc , \varepsilon _n$ of $\fh^\ast$ where $\varepsilon_{i}(\xi_{j}\partial_{j})=\delta_{i,j}$ for all $1 \leq i,j \leq n.$
Serganova determines a necessary and sufficient combinatorial condition for $\lambda$ to be typical. Namely, by \cite[Lemma 5.3]{Ser} one has that the set of atypical weights $\Omega_W$ for $W(n)$ is
 $$\Omega_W =\{a\varepsilon_i+\varepsilon_{i+1}+ \dotsb  +\varepsilon_n \in \fh^{*}\mid a\in \C,\ 1\leq i\leq n\}.$$
 and the set of atypical  weights  $\Omega$ for $\fg = \overline{S}(n)$ is
 $$\Omega =\{a\varepsilon_1+\dots.+a\varepsilon_{i-1}+b\varepsilon_i+(a+1)\varepsilon_{i+1}+ \dots +(a+1)\varepsilon_n\mid a, b \in \C, 1\leq i\leq n\}.$$

Let $\sigma =\varepsilon_1+...+\varepsilon_n$. For each $\la \in \Omega, \la \neq a\sigma $ there exists a unique $\overline{\la}=\la-a\sigma $
such that  $\overline{\la}$ is atypical for $W(n)$. Since $\dim L(a\sigma )=1$, we have
$$L(\la)\cong L(\overline{\la})\otimes L(a\sigma ).$$

\begin{thm}\label{T:Scutinvariants}  Let $\fg= \overline{S}(n)$ and let $p\geq 0$. Then, $\Lambda_{s}^p((\fg/\fg_0)^\ast)^{\fg_0}
\cong S^p(\fg_{-1}^\ast \oplus \fg_1^\ast )^{\fg_0}.$
\end{thm}

\begin{proof}

Since $\fg$ is a $\Z$-graded subalgebra of $W(n)$ and  $\fg_0 \cong W(n)_0$, the result follows from \cite[Theorem 4.2.1]{BAKN}.
\end{proof}

Now we can use Theorem~\ref{T:Scutinvariants} to show that the cohomology ring $\HH^\bullet (\fg,\fg_0;\C)$
is identified with a ring of invariants.

\begin{lem}\label{T:Scohomologyring} Let $\fg=\overline{S}(n)$. Let $G_{0}\cong \operatorname{GL}(n) $ be the connected reductive group such that $\Lie (G_{0}) =\fg_{0}$ and the adjoint action of $G_{0}$ on $\fg$ differentiates to the adjoint action of $\fg_{0}$ on $\fg.$ Then,
$$\HH^\bullet (\fg,\fg_0;\C)\cong S(
\fg_{-1}^\ast \oplus \fg_1^\ast )^{\fg_0}=S(
\fg_{-1}^\ast \oplus \fg_1^\ast )^{G_0}.$$
\end{lem}
\begin{proof} Recall that cochains are defined as follows
\begin{equation} \label{E:defco} C^p(\La, \La_0; \C) =  \Hom_{\La_0}(\Lambda_s^p(\La /\La_0), \C) \cong \Lambda_{s}^{p} ((\fg / \fg_{0})^{*} )^{\fg_{0}}.
\end{equation}
By Theorem ~\ref{T:Scutinvariants} we have
\begin{equation}\label{E:cutinvar1} \Lambda_{s}^{p} ((\fg / \fg_{0})^{*} )^{\fg_{0}} \cong \Lambda_{s}^{p} (\fg_{-1}^{*} \oplus \fg_{1}^{*} )^{\fg_{0}} .
\end{equation}
Combining (\ref{E:defco}) and (\ref{E:cutinvar1}), we have

\begin{equation}
C^p(\La, \La_0; \C) \cong \Lambda_{s}^{p} (\fg_{-1}^{*} \oplus \fg_{1}^{*} )^{\fg_{0}}.
\end{equation}
Next step is to show that differentials are identically zero and this can be done by arguing as in Proposition \ref{P:cohomologyring} by using  Theorem \ref{T:Scutinvariants}

\end{proof}

\begin{thm}\label{T:fg}
Let $M$ be a finite dimesional $\fg= \overline{S}(n)$-supermodule.
\begin{itemize}
\item [(a)]The superalgebra $\HH^\bullet (\fg,\fg_0;\C)$ is a finitely generated commutative  ring.
\item [(b)]The cohomology $\HH^\bullet (\fg,\fg_0;M)$ is finitely generated as an $\HH^\bullet (\fg,\fg_0;\C)$-supermodule.
\end{itemize}
\end{thm}
\begin{proof} This follows from  Lemma \ref{T:Scohomologyring} and Corollary \ref{T:Gfinitegeneration}.
\end{proof}

\subsection{}
Recall that the Lie superalgebra $\fg=\overline{S}(n)$ admits a
$\Z$-grading and  $\fg_0\cong \mathfrak{sl}(n)\oplus \C \cong \gl (n)$ as a Lie algebra.
Following \cite{BAKN}, we are interested in the natural problem of computing the relative cohomology
for the pair $(\fg, \fg_0)$.

Since by Lemma \ref{T:Scutinvariants}
\begin{equation}\label{E:Scohomiso}
\HH^{\bullet}(\fg ,\fg_{0};\C)\cong S^{\bullet}\left((\fg_{-1} \oplus \fg_{1})^{*} \right)^{\fg_{0}},
\end{equation}
   it is enough to compute $S^{\bullet}\left((\fg_{-1} \oplus \fg_{1})^{*} \right)^{\fg_{0}}$.  Since $\fg$ is a $\Z$-graded subalgebra of $W(n)$,  we will benefit from calculations done in \cite[Section 5]{BAKN}.

\subsection{}

Let $T \subseteq G_{0}$ denote the maximal torus consisting set of all diagional matrices  and $\fh  \subseteq  \fg_{0}$ to be
$\fh = \operatorname{Lie}\left(T \right),$ the Cartan subalgebra of $\fg_0$.

Recall that $\fg_{-1}$  has basis $\{\partial_{i}\mid 1\leq i\leq n\}$, $\fg_{0}$
has basis $\{ \xi_i\partial_j \mid 1\leq i, j \leq n\}$
and $\fg_{1} $ with basis $\{\xi_{i}\xi_{j}\partial_{k}\mid 1\leq i\neq j \neq k\leq n, i<j\}\cup \{\xi_1h_{2j}
\mid 3\leq j\leq n\}\cup \{\xi_ih_{1k}\mid  2\leq i\neq k\leq n \}$.

Let $\ff_{\1}$ be as in \cite[Lemma 5.5.2]{BAKN}, i.e., the $\C$-span of the vectors
$$\{\partial _1, \xi_1\xi_i\partial _i\mid i=2, \dotsc , n\}.$$
The intersection of $\ff_{\1}$ with $\fg_{-1} \oplus \fg_{1}$ is spanned by the vectors
\begin{equation}\label{E:a1}
\fa_{\1}:=\{\partial _1,\xi_1\xi_2\partial_2- \xi_1\xi_i\partial _i\mid i=3, \dotsc , n\}.
\end{equation}

Let $N$ denote the normalizer of $\fa_{\1}$ in $G_0$. By arguing as in \cite[Lemma]{BKN} one can show that  $$N=T\Sigma_{n-2}.$$

\subsection{}We can now explicitly describe the cohomology ring $\HH^\bullet(\fg, \fg_0; \C)$. Let $Z_k\in \fa_{\1}^{\ast}$ be given by
$Z_k(\xi_1\xi_2\partial_2-\xi_1\xi_i\partial_i)=\delta_{i,k}$ ($i,k=3, \dotsc, n$) and $Z_k(\partial_1)=0$. Let $\partial_1^\ast$ be given by
$\partial_1^\ast(\xi_1\xi_2\partial_2-\xi_1\xi_i\partial_i)=0$ for all $i=3,\dotsc, n$ and $\partial_1^{\ast}(\partial_1)=1$.

\begin{thm}\label{T:ScalculatingR} Restriction of functions defines an isomorphism,
\[
\HH^\bullet(\fg, \fg_0; \C) \cong S^{\bullet}\left((\fg_{-1} \oplus \fg_{1})^{*} \right)^{\fg_{0}} \cong S^{\bullet}(\fa_{\1}^{*})^{N} = \C[Z_{3}\partial_1^{\ast}, \dotsc , Z_{n}\partial_1^{\ast}]^{\Sigma_{n-2}},
\]
 where $\Sigma_{n-2}$ acts on $Z_{3}\partial_1^{\ast}, \dotsc , Z_{n}\partial_1^{\ast}$ by permuations.  Therefore, $R$ is a polynomial ring in $n-2$
variables of degree $2, 4, \dotsc , 2n-4.$
\end{thm}
\begin{proof} The first isomorphism is by Lemma \ref{T:Scutinvariants} and the second isomorphism follows from \cite[Corollary 4.4]{LR}.
Since $T$ is a normal subgroup of $N,$ we can first compute the $S^{\bullet}(\fa_{\bar{1}}^{*})^{T}=\C[Z_{3}\partial_1^{\ast}, \dotsc , Z_{n}\partial_1^{\ast}].$
It's straightforward to check that $\Sigma_{n-2}$ acts by permuting the variables $Z_{3}\partial_1^{\ast}, \dotsc , Z_{n}\partial_1^{\ast}.$  By a well known result on invariants under a symmetric group, it follows that $\HH^\bullet(\fg, \fg_0; \C)$ is a polynomial ring generated by elementary symmetric polynomials in the $Z_{i}\partial_{1}^{*}$, where degree of $Z_{i}\partial_{1}^{*}$ is two..

\end{proof}

\subsection{}
Let $\fa_{\1} \subset \fg_{\1}$ be as in (\ref{E:a1})  and let $\fa_{\0} = \Lie(T)= \fh \subseteq \fg_0$. One can easily verify that $\fa= \fa_{\0} \oplus \fa_{\1}$ is a Lie subsuperalgebra of $\fg$.
$$\HH^\bullet(\fa, \fa_{\0}; \C) \cong S^{\bullet}(\fa_{\1}^{*})^{\fa_{\0}} \cong S^{\bullet}(\fa_{\1}^{*})^T$$

\subsection {Detecting subalagebra for $\fg = \overline{S}(n)$} \label{S:detecting}

Let $\fe_{\1} = \fa_{\1}$ and let $\fe_{\0} = \Lie(T_{n-1})$, where
$$T_{n-1}:= \{\operatorname{diag}(t_1,  \dots, t_n) \in T\mid t_{1}= 1\}. $$ Set $\fe = \fe_{\0} \oplus \fe_{\1}$. One can directly verify that $\fe$ is a Lie subsuperalgebra of $\fg$.
The classical Lie superalgebra $\fe$ has the following important property

\begin{equation}\label{E:Stildefbrackets}
\left[\fe_{\0},\fe_{\0} \right]= \left[\fe_{\0},\fe_{\1} \right]=0.
\end{equation}

Since $\HH^{\bullet}(\fa,\fa_{\0};\C ) \cong \HH^{\bullet}(\fe,\fe_{\0};\C )^T$ and $\HH^\bullet (\fg,\fg_{0};\C ) \cong \HH^{\bullet}(\ff,\ff_{\0};\C )^{\Sigma_{n-2}}$ we have  the following isomorphism
$$\HH^\bullet (\fg,\fg_{0};\C ) \cong \HH^{\bullet}(\fe,\fe_{\0};\C )^{T\Sigma_{n-2}}.$$

From the discussion above one sees that $\fe$ is a detecting subalgebra for the pair $(\fg, \fg_0).$

Since by \eqref{E:Stildefbrackets} the structure of $\fe$ is of the type considered in \cite[Sections 5, 6]{BKN1}, \cite[Theorem 6.3.2]{BKN1} implies that one has a canonical isomorphism
\begin{equation}\label{E:rankiso}
\V _{\left(\fe,\fe_{\0} \right)}(M) \cong \V_{\fe}^{\rank}(M)
\end{equation}
for any finite dimensional $\fe$-supermodule $M$ which is an object of $\mathcal{C}_{(\fe,\fe_{\0})}.$ We identify the rank and support varieties of $\fe$ via this isomorphism.



\subsection{ Support varieties of simple modules} 

 Since $\HH^\bullet (\fe,\fe_{\0};\C ) \cong S(\fe_{\1}^\ast)^{\fe_{\0}}= S(\fe_{\1}^\ast) \cong \C[\partial_1^\ast, Z_2, \dots, Z_n]$  , we can  identify the support variety of any finite dimensional module $M \in \mathcal{C}_{(\fe,\fe_{\0})}$ with the conical affine subvariety of the affine $(n-1)$-space
\[
\operatorname{MaxSpec}\left( \HH^\bullet (\fe,\fe_{\0};\C ) \right) = \V_{(\fe,\fe_{\0})}(\C ) \cong \mathbb{A}^{n-1}
\] defined by the ideal $I_{(\fe,\fe_{\0})}(M).$

The inclusion $\fe \hookrightarrow \fg$ induces a restriction map on cohomology which, in turn, induces maps of support varieties.  That is, given supermodules $M$ and $N$ in $\mathcal{C}_{(\fg,\fg_{0})}$ one has $M \in \mathcal{C}_{(\fe,\fe_{\0})}$ by restriction to $\fe$ and one has maps of varieties
\begin{align*}
\resstar &:\V_{(\fe, \fe_{\0})}(M,N) \to \V_{(\fg, \fg_{0})}(M,N),\\
\resstar &:\V_{(\fe, \fe_{\0})}(M) \to \V_{(\fg, \fg_{0})}(M).
\end{align*}
Similarly the inclusion $\fe \hookrightarrow \ff$ induces the following maps of varieties
\begin{align*}
\resstar &:\V_{(\fe, \fe_{\0})}(M,N) \to \V_{(\fa, \fa_{\0})}(M,N),\\
\resstar &:\V_{(\fe, \fe_{\0})}(M) \to \V_{(\fa, \fa_{\0})}(M).
\end{align*}
for finite dimensional modules $M, N \in \mathcal{C}_{(\fa,\fa_{\0})}$.

By arguing exactly as in \cite[Theorem 6.4.1]{BAKN} one proves that
\begin{equation}\label{E:suppae}
\V_{(\fa, \fa_{\0})}(M) \cong \V_{(\fe, \fe_{\0})}(M)/T
\end{equation}

for any finite dimensional module $M \in \mathcal{C}_{(\fa,\fa_{\0})}$.

\subsection{} In this subsection we compute support varieties of all simple finite dimensional $\fg$-modules. The first step in this calculation is to compute support variety of Kac modules. 
Our results show that $L(\lambda)$ is typical if and only if the support variety of $L(\lambda)$ is zero.
\begin{prop}\label{Kacsupport:P}  Let $\la \in X_{0}^{+}$, $K(\la)$ be the associated  Kac supermodule. Then  $$\V_{(\fg, \fg_0)}(K(\la))=\{0\}.$$
\end{prop}
\begin{proof}It is enough to
show that $\Ext^{n}_{\mathcal{C}_{(\fg,\fg_{0})}}(K(\lambda), K(\la))=0$, for $n>>0$.

By Frobenius reciprocity, for all $n$ we have
$$\Ext^{n}_{\mathcal{C}_{(\fg,\fg_0)}}(K(\lambda),K(\la))\cong \Ext^{n}_{\mathcal{C}_{(\fg_{0}\oplus
\fg^{+},\fg_0)}}( L_{0}(\lambda),K(\la)).$$
Since $\fg^{+}$ is an ideal in $\fg_{0}\oplus\fg^{+}$ one can apply the Lyndon-Hochschild-Serre spectral sequence to $(\fg^{+}, \{0 \}) \subseteq (\fg_{0}\oplus \fg^{+},\fg_{0})$:
$$E_{2}^{i,j}=\Ext^{i}_{\mathcal{C}_{(\fg_0,\fg_0)}}(L_{0}(\lambda),
\Ext^{j}_{\mathcal{C}_{(\fg^{+},\{0\})}}({\mathbb C}, K(\la)))\Rightarrow
\Ext^{i+j}_{\mathcal{C}_{(\fg_{0}\oplus \fg^{+},\fg_0)}}(L_{0}(\lambda), K(\la)).$$

Since ${\mathcal{C}_{(\fg_{0},\fg_{0})}}$ consists of $\fg_{0}$-supermodules which are finitely semisimple over $\fg_{0},$ this spectral sequence is zero for $i >0.$ That is, it collapses
at the $E_{2}$ page and yields
\begin{equation}\label{E:iso}
\Hom_{\fg_{0}}(L_{0}(\lambda), \Ext^{n}_{\mathcal{C}_{(\fg^{+},\{0\})}}({\mathbb C}, K(\la))
\cong \Ext^{n}_{\mathcal{C}_{(\fg_{0}\oplus \fg^{+},\fg_0)}}(L_{0}(\lambda),K(\la)).
\end{equation}
Now one can argue as in the proof of the \cite[Proposition 7.1.1]{BAKN}
\end{proof}

Recall that superdimension of a  module $M$ is
defined to be $ \dim M_{\0} -\dim M_{\1}$ and  the set of atypical weights for $\fg$ is given by
$$\Omega =\{a\varepsilon_1+\dots.+a\varepsilon_{i-1}+b\varepsilon_i+(a+1)\varepsilon_{i+1}+ \dots +(a+1)\varepsilon_n\mid a, b \in \C, 1\leq i\leq n\}.$$

\begin{thm} Let $\la \in X_{0}^{+}$ and let  $L(\la)$ be the finite dimensional simple $\fg$-supermodule
with highest weight $\la$. Then
\begin{itemize}
\item[(a)] If $\la \notin \Omega$ then $\V_{(\fg, \fg_0)}(L(\la))=\{0\}$.
\item[(b)] If $\la \in \Omega$ then $\V_{(\fg, \fg_0)}(L(\la))=\V_{(\fg, \fg_0)}(\C)$.
In this case
the support variety has dimension $n-2$.
\end{itemize}
\end{thm}
\begin{proof}
\begin{itemize}

\item[(a)] If $\la \notin \Omega$, i.e., $\la$ is typical, then by \cite[Theorem 6.3]{Ser} $L(\la) = K(\la)$.
Since $\V_{(\fg, \fg_0)}(K(\la))=\{0\}$ by Proposition \ref{Kacsupport:P},  the result follows.

 \item[(b)] First observe that
$$\Omega \cap X_{0}^{+}=\{a\varepsilon_{1}+a\varepsilon{_2}\dots +a\varepsilon _{n-1}+b\varepsilon_n \mid a,b\in {\Z}, b\leq a\}. $$
Furthermore, it is enough to prove  that if $\la \in \Omega\cap X_{0}^{+}$ then
$\V_{({\fe},{\fe}_{\0})}(L(\la))=\V_{({\fe},{\fe}_{\0})}(\C)$. Then by (\ref{E:suppae}) we have
$$ \V_{(\fa, \fa_{\0})}(L(\lambda))\cong \V_{(\fe, \fe_{\0})}(L(\lambda))/T = \V_{(\fe, \fe_{\0})}(\C)/T \cong \V_{(\fa, \fa_{\0})}(\C) $$
this in turn gives

\[
\V_{(\fg, \fg_{0})}(\C) = \resstar \left( \V_{(\fa, \fa_{\0})}(\C )\right)=\resstar \left( \V_{(\fa, \fa_{\0})}(L(\lambda))\right) \subseteq \V_{(\fg, \fg_0)}(L(\la)) \subseteq  \V_{(\fg, \fg_0)}(\C),
\]
which implies the result for $\fg.$

There are two cases we need to consider: $\la = a\sigma$ and $\la \neq a\sigma$, where $\sigma =\varepsilon_1+...+\varepsilon_n$.

If $\la=a\sigma$, since $\dim L(a\sigma)=1$, super dimension of  $L(\la)= \dim L(\la)_{\0} - \dim L(\la)_{\1} $ is nonzero, then by Lemma \ref{L:super} (4) we have
$$\V_{({\fe},{\fe}_{\0})}(L(\la))=\V_{({\fe},{\fe}_{\0})}(\C).$$

Now if $\la \neq  a\sigma$, then
there exists an atypical weight $\bar{\la}$ for $W(n)$ such that $L(\la) \cong L(\bar{\la})\otimes L(a\sigma)$. We also know from the work of Serganova \cite{Ser} that
  a simple finite dimensional module for $W(n)$ is typical if and only if its superdimension is zero. Since $\bar{\la}$ is atypical for $W(n)$, super dimension of $ L(\bar{\la}) $ is nonzero and thus super dimension of $ L(\la) $ which is equal to  superdimension of $L(\bar{\la}) $ is nonzero.
  Now again from  Lemma \ref{L:super} (4) it follows that
$$\V_{({\mathfrak e},{\fe}_{\0})}(L(\la))=\V_{({\fe},{\fe}_{\0})}(\C).$$

This completes the proof.
\end{itemize}
\end{proof}

\subsection{} 

In this subsection we show that a realization theorem holds for $\fg$.
Recall that $\mathcal{C}_{(\fg, \fg_0)}$ denotes the category of $\fg$-modules which are finitely semisimple as $\fg_0$-module
and $L_{\zeta}$ denotes the Carlson module corresponding to the homogeneous element $\zeta \in \HH^{\bullet}(\fg, \fg_0; \C)$. Since the Carlson modules may not be finite dimensional  we are going to work with relative support varieties.

\begin{thm} Let $X$ be a conical subvariety of ${\mathcal V}_{(\fg, \fg_0)}(\C)$.
Then there exists a $\fg$-module  $ M \in \mathcal{C}_{(\fg, \fg_0)}$ such that
$$ \V_{(\fg, \fg_0)}(M,\C)=X.$$
\end{thm}
\begin{proof} By Theorem \ref{T:fg} we know that  $\HH^{\bullet}(\fg, \fg_0; \C)$ is finitely generated. We have seen in subsection \ref{S:detecting} that the pair $(\fg, \fg_0)$ admits a detecting subalgebra $\fe = \fe_{\0} \oplus \fe_{\1}$ and $\fe-$ support varieties have rank variety description. Thus the pair $(\fg, \fg_0)$ satisfies all assumptions of Theorem \ref{T:Realization} and the result follows.
\end{proof}

\subsection{Acknowledgements}The author would like to thank his PhD advisor Daniel Nakano for helpful discussions.

\end{document}